\documentclass[reqno, a4paper]{amsart}

\usepackage{amsmath}
\usepackage{txfonts}
\usepackage{amsfonts}
\usepackage{amssymb}
\usepackage{graphicx}
 \usepackage[sort&compress]{natbib}
\newtheorem{theorem}{Theorem}[section]
\theoremstyle{plain}

\newtheorem{conjecture}[theorem]{Conjecture}
\newtheorem{question}[theorem]{Question}
\newtheorem{corollary}[theorem]{Corollary}
\newtheorem{lemma}[theorem]{Lemma}
\newtheorem{proposition}[theorem]{Proposition}

\theoremstyle{definition}
\newtheorem{remark}[theorem]{Remark}

\numberwithin{equation}{section}
\newcommand\R{{\mathbb R}}
\newcommand\Q{{\mathbb Q}}
\newcommand\N{{\mathbb N}}
\newcommand\Z{{\mathbb Z}}

\newcommand\mca{{\mathcal A}}
\newcommand\mcb{{\mathcal B}}
\newcommand\mcc{{\mathcal C}}

\newcommand\rect{{\mathcal S}}
\newcommand\irange{{\mathcal I}}

\DeclareMathOperator*{\orb}{orb}
\DeclareMathOperator*{\area}{area}
\DeclareMathOperator*{\id}{id}
\DeclareMathOperator*{\inter}{int}

\providecommand{\q}[1]{\lq#1\/\rq}

\begin{document}
\title[Dynamics of a Continuous Piecewise Affine Map of the Square]{Dynamics of a Continuous Piecewise Affine Map of the Square}
\author{Georg Ostrovski}
\address{Mathematics Department, University of Warwick}
\email{g.ostrovski@gmail.com}

\begin{abstract}
We present a one-parameter family of continuous, piecewise affine, area preserving maps of the 
square, which are inspired by a dynamical system in game theory. Interested in the coexistence 
of stochastic and (quasi-)periodic behaviour, we investigate invariant annuli separated by invariant 
circles. For certain parameter values, we explicitly construct invariant circles both of rational and
irrational rotation numbers, and present numerical experiments of the dynamics on 
the annuli bounded by these circles.
\end{abstract}

\maketitle

\section{Introduction}
Piecewise affine maps and piecewise isometries have received a lot of attention, either as simple, 
computationally accessible models for complicated dynamical behaviour, or as a class of systems 
with their own unique range of dynamical phenomena. For a list of examples, see 
\cite{Bullett1986,Wojtkowski1980,Wojtkowski1981,Wojtkowski2008,Przytycki1983,
Devaney1984,Aharonov1997,Gutkin1995,Goetz2000,Adler2001,Lagarias2005a,Ashwin2005,Beardon,Reeve-Black2013,Ostrovski2010} and references therein.

In this paper, we investigate a family of piecewise affine homeomorphisms of the unit square, 
motivated by a particular class of dynamical systems 
modeling learning behaviour in game theory, the so-called Fictitious Play dynamics 
(see \cite{Sparrow2007,VanStrien2010,VanStrien2011}).
It was shown in \cite{VanStrien2010} that this dynamics can be represented by a flow on $S^3$ with a topological disk 
as a global first return section, whose first return map is \emph{continuous, piecewise affine, area preserving
and fixes the boundary of the disk pointwise}. For an investigation of the itinerary 
structure and some numerical simulations of such first return maps, see \cite{Ostrovski2010}.

The maps considered in this paper are chosen to satisfy these properties, 
and the nine-piece construction considered here 
seems to be the simplest possible (nontrivial) example satisfying all of them.
Its qualitative behaviour resembles that seen in 
\cite{Ostrovski2010} for the first return maps of Fictitious Play; in particular, the ways in which
stochastic and (quasi-)periodic behaviour coexist seem to be of similar type, giving rise to similar
phenomena.

In Sections \ref{sec:construction} and \ref{sec:properties} 
we present a geometric construction of our family of maps
and describe its basic formal properties. Then, interested in the long-term behaviour of iterates 
of the maps, our next goal is to establish the existence of certain invariant regions.
For that, in Section \ref{sec:inv_circles}
we develop some technical results about periodic orbits and invariant curves, and in Section \ref{sec:spec_vals}
prove their existence for certain parameter values. 
Finally, in  Section \ref{sec:gen} we discuss the dynamics for more general parameter values, 
present numerical observations and discuss open questions.

\section{Construction of the map}\label{sec:construction}

Let us denote the unit square by $\rect = [0,1]\times [0,1]$. 
We construct a one-parameter family of continuous, piecewise affine maps
$F_\theta \colon \rect \to \rect$, $\theta \in (0,\frac{\pi}{4})$, 
as follows (see Fig.\ref{fig:construction} for an illustration). 

Denote the four vertices of $\rect$ by $E_1 = (0,0)$, $E_2 = (1,0)$, $E_3 = (1,1)$ and $E_4 = (0,1)$.
In the following we will use indices $i\in\irange = \{1,2,3,4\}$ with cyclic order, i.e., with the understanding
that index $i+1$ is $1$ for $i=4$ and index $i-1$ is $4$ for $i=1$.

Let $\theta \in (0,\frac{\pi}{4})$, and for $i \in\irange$ let $L_i$ be the ray through $E_i$,
such that the angle between the segment $\overline{E_i E_{i+1}}$ and $L_i$ 
is $\theta$. Let $P_i \in \inter(\rect)$ be the point $L_{i-1} \cap L_i$, then
the $P_i$, $i \in\irange$, form a smaller square inside $\rect$. 
Now we divide $\rect$ into the following nine regions (see Fig.\ref{fig:construction}(left)):
\begin{itemize}
 \item four triangles $\mca_i = \Delta(E_i,E_{i+1},P_i)$, $i\in\irange$, each adjacent to one of the 
sides of $\rect$;
\item four triangles $\mcb_i = \Delta(E_i,P_i,P_{i-1})$, $i\in\irange$, each sharing one side with 
$\mca_{i-1}$ and $\mca_i$;
\item a square $\mcc = \square(P_1,P_2,P_3,P_4)$, each side of which is adjacent to one of the $\mcb_i$.
\end{itemize}

Now, we repeat the same construction 'in reverse orientation', to obtain a second, very similar partition of $\rect$, 
as shown in Fig.\ref{fig:construction}(right). Here we denote the vertex of the inner square which has the same y-coordinate 
as $P_1$ by $P'_1$, and the other vertices of the inner square by $P'_2$, $P'_3$, $P'_4$, in counterclockwise order.
For $i\in\irange$ we denote the triangles $\mca'_i = \Delta(E_i, E_{i+1}, P'_i)$, $\mcb'_i = \Delta(E_i, P'_i, P'_{i-1})$
and the square $\mcc' = \square(P'_1,P'_2,P'_3,P'_4)$.

Finally, the map $F = F_\theta \colon \rect \to \rect$ is uniquely defined by the data
\begin{itemize}
 \item $F(E_i) = E_i$, $i\in\irange$;
 \item $F(P_i) = P'_i$, $i\in\irange$;
 \item $F$ affine on each of the pieces $\mca_i$, $\mcb_i$, and $\mcc$.
\end{itemize}

\begin{figure}
\centering
\includegraphics[width=\textwidth]{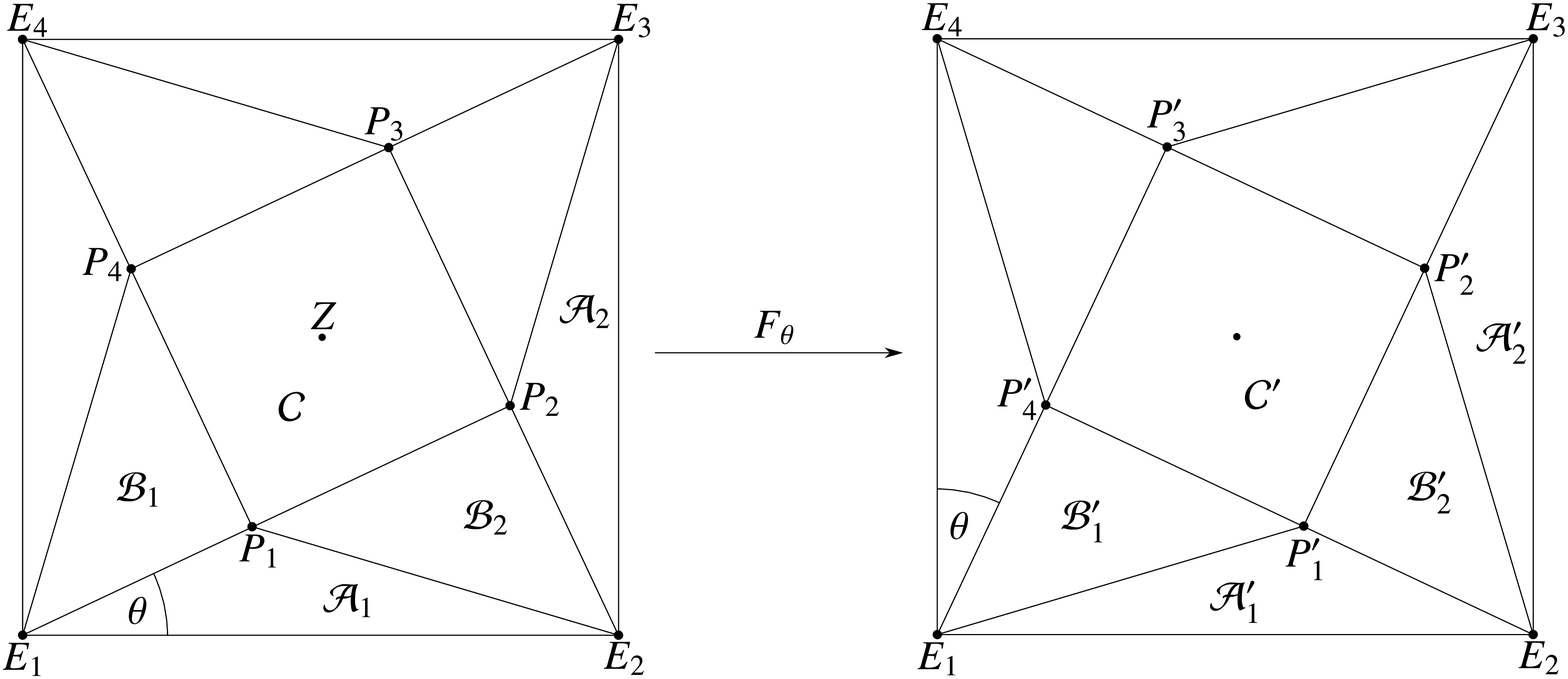}
\caption{Construction of the map $F = F_\theta$: $F(E_i) = E_i$, $F(P_i) = P'_i$;
$F$ is affine on each of $\mca_i$, $\mcb_i$ and $\mcc$, 
such that $F(\mca_i) = \mca'_i$, $F(\mcb_i) = \mcb'_i$ and $F(\mcc) = \mcc'$.}
\label{fig:construction}
\end{figure}

\section{Properties of $F$}\label{sec:properties}

It is easy to see that $F$ is a piecewise affine 
homeomorphism of $\rect$, with $F(\mca_i) = \mca'_i$ and $F(\mcb_i) = \mcb'_i$ for each $i\in\irange$, 
and $F(\mcc) = \mcc'$.
Moreover, $F$ is area and orientation preserving, 
since $d F$ is constant on each of the pieces, with $\det dF = 1$ everywhere.
Note that while $F$ is continuous, its derivative $dF$ has discontinuity lines along the boundaries
of the pieces; we call these the \emph{break lines}. 
Note also that $F\vert_{\partial \rect} = \id$.

We denote $P_1 = (s,t)$, where $t \in (0,\frac{1}{2})$ and $s \in (0,\frac{1}{2})$.
The coordinates of the other points $P_i$ and $P'_i$ are then given by symmetry. Simple geometry gives that 
$s$, $t$ and $\theta$ satisfy
\begin{equation}\label{eq:relations}
 t - t^2 = s^2, \quad t = \sin^2 \theta, \quad s = \sin \theta \cos \theta.
\end{equation}

The map $F$ is given by three types of affine maps $A$, $B$, and $C$:
\begin{itemize}
 \item  $A \colon \mca_1 \to \mca'_1$ is a shear fixing 
$\overline{E_1 E_2} = [0,1]\times\{0\}$ and mapping $P_1$ to $P'_1$:
\begin{equation*}
 A(x,y) = \begin{pmatrix} 1 & \frac{1-2s}{t} \\ 0 & 1 \end{pmatrix} \begin{pmatrix} x \\ y \end{pmatrix}.
\end{equation*}
It leaves invariant horizontal lines and moves points 
in $\inter(\mca_1)$ to the right (since $(1-2s)/t>0$).

\item $B \colon \mcb_1 \to \mcb'_1$ is a linear scaling 
map with a contracting and an expanding direction,
defined by $B(E_1) = E_1$, $B(P_1) = P'_1$ and $B(P_4) = P'_4$: 
\begin{equation*}
B(x,y) = \frac{1}{t-s} \begin{pmatrix}t^2-(1-s)^2 & (1-2s)t \\ (2s-1)t & (2t-1)t\end{pmatrix} 
\begin{pmatrix} x \\ y \end{pmatrix}.
\end{equation*}
It can be checked that the contracting direction of $B$ 
lies in the sector between $\overline{E_1 P_4}$ and $\overline{E_1 E_4}$, 
and the expanding direction in the sector betwen $\overline{E_1 E_2}$ and $\overline{E_1 P_1}$.
From general theory we also have that $B$ preserves the quadratic form given by 
\begin{equation}\label{eq:quad_form_B}
 Q_B(x,y) = t (x^2 + y^2) - x y . 
\end{equation}

\item $C \colon \mcc \to \mcc'$ is the rotation about 
$Z = (\frac{1}{2},\frac{1}{2})$, mapping $P_i$ to $P'_i$, $i \in\irange$: 
\begin{equation*}
 C(x,y) = \begin{pmatrix}2s & 2t-1 \\ 1-2t & 2s\end{pmatrix} \begin{pmatrix} x \\ y \end{pmatrix}  
+ \begin{pmatrix} 1-t-s \\ t-s \end{pmatrix}.
\end{equation*}
The rotation angle is $\alpha = \frac{\pi}{2} - 2 \theta$, where $\theta$ is the parameter angle
in the construction of the map $F$.
\end{itemize}
All other pieces of $F$ are analogous, by symmetry of the construction. 
To capture this high degree of symmetry, we make the following observations
which follow straight from the definition of $F$.

\begin{lemma} \label{lem:rot_sym}
 Let $R$ denote the rotation about $Z$ by the angle $\frac{\pi}{2}$. Then $F$ and $R$ commute:
\begin{equation*}
 F \circ R = R \circ F .
\end{equation*}
\end{lemma}

\begin{lemma}
 Let $S(x,y) = (y,x)$. Then $F$ is \textit{$S$-reversible}, i.e., $S$ conjugates $F$ and $F^{-1}$: 
\begin{equation*}
 S \circ F \circ S = F^{-1}.
\end{equation*}
Further, $F$ is $T_1$- and $T_2$-reversible for the 
reflections $T_1 (x,y) = (1-x,y)$ and $T_2(x,y) = (x,1-y)$.
\end{lemma}

Heuristically, $F$ acts similarly to a twist map: The iterates $F^n(X)$ 
of any point $X\in \inter(\rect)$ rotate counterclockwise 
about $Z$ as $n \to \infty$. The 'rotation angle' 
(the angle between $\overline{Z X}$ and $\overline{Z F(X)}$)
 is not constant, 
but it is bounded away from zero as long as $X$ is bounded away
from $\partial \rect$; in particular, every point whose orbit stays bounded away from the boundary
runs infinitely many times around the centre. 

Note also that the rotation angle is
monotonically decreasing along any ray emanating from the centre $Z$. 
However, it is not strictly 
decreasing, as all points in $\mcc$ rotate by the same angle; this
sets the map $F$ apart from a classical twist map, for which strict monotonicity 
(the \q{twist condition}) is usually required.

\section{Invariant Circles}\label{sec:inv_circles}

Clearly, the circle inscribed to the inner square $\mcc$ and all concentric circles 
in it centred at $Z$ are invariant under $F$, which acts as a rotation on these circles. 
When $\theta$ is a rational multiple 
of $\pi$, the rotation $C$ is a rational (periodic) rotation, and a whole regular $n$-gon 
inscribed to $\mcc$ is $F$-invariant, see Fig.\ref{fig:firstparam}. 

We are interested in other invariant circles encircling $Z$, 
as these form barriers to the motion of points under $F$ and provide a partitioning 
of $\rect$ into $F$-invariant annuli.
Numerical simulations indicate that such curves exist for many parameter 
values $\theta$ and create invariant annuli, on which the motion is
predominantly stochastic. 

This section follows closely the arguments of Bullett \cite{Bullett1986}, 
where in a similar way, invariant circles are studied for a piecewise linear version of the standard map.
The idea is to study the orbits of the points where the invariant circles intersect break lines
and to prove that these follow a strict symmetry pattern, forming so-called cancellation orbits.

We consider invariant circles $\Gamma$ on which $F$ 
preserves the $S^1$-order of points. 
This, for example, is the case if all rays from $Z$ intersect 
the circle $\Gamma$ in precisely one point. 
For an invariant circle $\Gamma$, we denote the rotation number 
of $F \vert_\Gamma \colon \Gamma \to \Gamma$ 
by $\rho_\Gamma = \rho(F \vert_\Gamma)$.
By simple geometric considerations, we get the following lemma.

\begin{lemma}
Let $\Gamma_1, \Gamma_2$ be two invariant circles for $F$ encircling $Z$. 
If $\Gamma_1$ is contained in the component of $\rect \setminus \Gamma_2$ 
containing $Z$, then $\rho_{\Gamma_1} \geq \rho_{\Gamma_2}$. 
\end{lemma}

In other words, if there is a family of such nested invariant circles, 
their rotation number is monotonically decreasing as the circles approach $\partial \rect$. 
It also follows that the rotation number $\rho$ of any orbit is bounded above
by the rotation number on the centre piece $\mcc$, i.e., 
\[0 \leq \rho \leq \frac{\alpha}{2 \pi} = \frac{1}{4}-\frac{\theta}{\pi}.\]

We now consider $F$-invariant circles near the boundary $\partial \rect$, 
which do not intersect the centre piece $\mcc$. 
Any such curve $\Gamma$ intersects exactly two types of break line segments:
the segments $\mcb_i \cap \mca_i = \overline{E_i P_i}$ 
and $\mca_i \cap \mcb_{i+1} = \overline{E_{i+1} P_{i+1}}$, $i\in\irange$.
Let us call these intersection points $U_i = \Gamma \cap (\mcb_i \cap \mca_i)$ 
and $V_i = \Gamma \cap (\mca_i \cap \mcb_{i+1})$.

We say that an invariant curve which encircles $Z$ and on which $F$ preserves the $S^1$-order 
is \emph{rotationally symmetric}, 
if it is invariant under the rotation $R$ (cf.~Lemma \ref{lem:rot_sym}): $R(\Gamma) = \Gamma$.

In the remainder of this section we will use the fact that any invariant circle $\Gamma$
with rational rotation number is of one of the following two types \cite{Katok1983}: 
\begin{itemize}
 \item pointwise periodic, that is, $F \vert_\Gamma$ is conjugate to a rotation of the circle;
 \item non-periodic, that is, $F \vert_\Gamma$ is not conjugate to a rotation; however, in this case, 
 $F \vert_\Gamma$ still has at least one periodic orbit. 
\end{itemize}

Following the ideas in \cite{Bullett1986}, we now prove a number of results illustrating
the importance of the orbits of $U_i$ and $V_i$ for the invariant circle containing them. 

\begin{lemma}\label{lem:F_per_case}
Let $\Gamma$ be a rotationally symmetric invariant circle disjoint from $\mcc$. Assume that $F\vert_\Gamma$ 
has rational rotation number $\rho_\Gamma = \frac{p}{q} \in \Q$, and that $F\vert_\Gamma$ is periodic. 
Then for $i\in\irange$ the orbit $\orb(U_i) = \{F^n(U_i)\colon n \in \Z \}$ contains some $V_j$ , 
$j \in\irange$, and vice versa. 
\end{lemma}

\begin{proof}
By symmetry it is sufficient to show the result for $U_1$. 
Suppose for a contradiction that $V_j \notin \orb(U_1)$ for all $j\in\irange$.
Let $\mathcal{O} = (\bigcup_i \orb(U_i)) \cup (\bigcup_j \orb(V_j))$, which by periodicity of $F$ is finite. 
Let $X,Y \in \mathcal{O}$ be the two points closest to $U_1$ on either side along $\Gamma$. 
Consider the line segment $\overline{X Y}$ which crosses the break line going through $U_1$. 
Then $F (\overline{X Y})$ is a \q{bent} line, consisting of two straight line segments, 
so that $F$ maps the triangle $\Delta(X,U_1,Y)$ to a quadrangle. 

Now, by assumption and periodicity of $F$, there exists $k > 1$ such that $F^k(U_1) = U_j$ for some $j \in \irange$ and 
$F^l(U_1) \notin \{U_i: i \in \irange \} \cup \{V_j: j \in \irange\} $ for $1 \leq l < k$. 
This implies that the triangle $\Delta(F(X),F(U_1),F(Y))$ is 
mapped by $F^{k-1}$ to the triangle $\Delta(F^k(X),F^k(U_1),F^k(Y))$ without
bending any of its sides (since $X$ and $Y$ are the points in $\mathcal{O}$ closest to $U_1$). 

By symmetry we have that $\Delta(F^k(X),F^k(U_1),F^k(Y)) = \tilde R (\Delta(X,U_1,Y))$,
where $\tilde R$ is the rotation about $Z$ by one of the angles $0, \frac{\pi}{2}, \pi, \frac{3\pi}{2}$, and hence
\begin{equation*}
 \area(\Delta(F^k(X),F^k(U_1),F^k(Y))) = \area(\Delta(X,U_1,Y)).
\end{equation*}
But $F$ is area preserving, so 
\begin{equation*} 
\area(\Delta(F^k(X),F^k(U_1),F^k(Y))) = \area(\Delta(F(X),F(U_1),F(Y))),
\end{equation*}
which is a contradiction because $F$ maps $\Delta(X,U_1,Y)$ 
to a quadrangle which either properly contains or is properly contained in $\Delta(F(X),F(U_1),F(Y))$.
\end{proof}

With a slightly bigger effort, we can extend the result to the case of non-periodic $F\vert_\Gamma$.

\begin{lemma}\label{lem:F_nonper_case}
Let $\Gamma$ be a rotationally symmetric invariant circle disjoint from $\mcc$. 
Assume that $F\vert_\Gamma$ 
has rational rotation number $\rho_\Gamma = \frac{p}{q} \in \Q$, and that $F\vert_\Gamma$ is not periodic. 
Then for $i\in\irange$, the orbit of $U_i$ contains some $V_j$ , 
$j \in\irange$, and vice versa. 
\end{lemma}

\begin{proof}
As in the previous lemma, we give a proof for $U_1$, 
the other cases following by symmetry. We distinguish two cases:

\textbf{Case 1: $U_1$ non-periodic for $F$.}\newline
Let us write $z_k = F^k(U_1)$ for $k \in \Z$.
Since $\rho_\Gamma = \frac{p}{q} \in \Q$, there exist points $Q$ and $Q'$ in $\Gamma$, 
each periodic of period $q$, such that
$z_{nq} \to Q$ and $z_{-nq} \to Q'$ as $n \to \infty$. 
Note that $F^q$ is affine in a sufficiently small neighbourhood on either side of $Q$. 
Then for sufficiently large $N$, the points $z_{nq}$, $n > N$, lie on the straight line segment 
$\overline{z_{Nq} Q}$ (the contracting direction at $Q$). 
Hence for large $n$, $\Gamma$ contains the straight line segment $\overline{z_{nq} Q}$. 
Analogously, for large $n$, the straight line segment $\overline{Q' z_{-nq}}$ is contained in $\Gamma$. 

In particular, $\ell = \overline{z_{-(m+2)q} z_{-mq}}$ is contained in $\Gamma$ for large $m$. 
But $U_1 \in F^{(m+1)q}(\ell)$, so $F^{nq}(\ell)$ has a kink for large $n$ unless $z_{Nq} = V_j$ for some $N$ and $j$ 
(note that since $U_1$ is non-periodic, $U_i \notin \orb(U_1)$, $i \in \irange$). 
Since $F^{nq}(\ell)$ is near $Q$ for large $n$, it has to be straight, 
and it follows that $V_j \in \orb(U_1)$ for some $j$.

\textbf{Case 2: $U_1$ periodic for $F$.}\newline
In this case $F^q(U_1) = U_1$ and the argument is similar to the proof of Lemma \ref{lem:F_per_case}.

Assume for a contradiction that $V_j \notin \orb(U_1)$ for all $j$. 
By symmetry, this implies $V_j \notin \bigcup_i \orb(U_i)$. 
Pick $X,Y \in \bigcup_i \orb(U_i)$ nearest to $V_1$ from each side
and denote by $S$ the segment of $\Gamma$ between $X$ and $Y$.
Since the straight line segment $\overline{X Y}$ crosses the break line which contains $V_j$, 
its image $F(\overline{X Y})$ has a kink. 
Therefore, since $F$ is area preserving, the area between $F(S)$ and $\overline{F(X) F(Y)}$ is either 
greater or less than the area between $S$ and $\overline{X Y}$. For $0 \leq k < q$, 
the area between $F^k(S)$ and $\overline{F^k(X) F^k(Y)}$ is equal to the area between
 $F^{k+1}(S)$ and $\overline{F^{k+1}(X) F^{k+1}(Y)}$, unless $F^k(S)$ contains one of the $V_j$. 
Whenever $V_j \in F^k(S)$ for some $j$ and $k$ 
(which can happen at most four times for $0 \leq k < q$), 
this area decreases or increases. By symmetry, these up to four changes have the same form, 
so the area either always decreases or always increases.
Note on the other hand that $F^q(S) = S$, $F^q(X) = X$ and 
$F^q(Y) = Y$ (since $X$ and $Y$ are in the $q$-periodic orbit of $U_1$). 
So if we denote the region between $S$ and $\overline{X Y}$ by 
$\Omega$, $\area(F^q(\Omega)) \neq \area(\Omega)$, which contradicts 
the fact that $F$ is area preserving. This finishes the proof.
\end{proof}

Combinining the above lemmas, we get the following result. 

\begin{proposition}\label{prop:rot_num}
Let $\Gamma$ be a rotationally symmetric invariant circle disjoint from $\mcc$
with rotation number $\rho_\Gamma = \frac{p}{q} \in \Q$. 
Then for every $i\in\irange$, the $F$-orbit of $U_i$ contains some $V_j$, 
$j \in\irange$, and vice versa.
Moreover, every such orbit contains an equal number 
$n$ of the $U_i$ and $V_j$, which are traversed in alternating order. 
If $n \geq 2$, then any such orbit is periodic. 
\end{proposition}

\begin{remark} \label{rem:canc_orbits}
In \cite{Bullett1986}, Bullett coined the term \q{cancellation orbits} 
for these orbits of break points on an invariant circle, 
reflecting the insight that each \q{kink} introduced by the discontinuity of $d F$
at one such point needs to be \q{cancelled out} by an appropriate \q{reverse kink} at another 
$d F$-discontinuity point, if the invariant circle has rational rotation number.

However, cancellation orbits can also occur on invariant circles with $\rho \notin \Q$. 
In that case these orbits are not periodic, 
and each cancellation orbit would only contain 
one of the $U_i$ and one of the $V_j$. We will see 
examples for this in Section \ref{sec:spec_vals}.
\end{remark} 

We now show that cancellation orbits determine the behaviour of the whole map $F\vert_\Gamma$.

\begin{proposition}
Let $\Gamma$ be a rotationally symmetric invariant circle disjoint from $\mcc$.
Then $F\vert_\Gamma$ is periodic if and only if the $U_i$- and $V_i$-orbits are periodic. 
\end{proposition}

\begin{proof}
Of course, if the rotation number $\rho_\Gamma$ is irrational, 
neither $F\vert_\Gamma$ nor the break point orbits
can be periodic, so we only need to consider rational rotation number. 

Further, if $F\vert_\Gamma$ is periodic, so are all cancellation orbits. 
For the converse, suppose for a contradiction that $U_i$ and $V_i$, $i \in \irange$, are periodic, 
but $F\vert_\Gamma$ is not.
We repeat an argument already familiar from the proof of Lemma \ref{lem:F_nonper_case}.

Pick any non-periodic point $P \in \Gamma$, then there exists $Q \in \Gamma$, 
such that $F^{nq}(P)\to Q$ as $n \to \infty$. 
Note that $F^q$ is affine in a sufficiently small neighbourhood on either side of $Q$. 
Then for $n > N$ sufficiently large, the points $F^{nq}(P)$ 
lie on a straight line segment from $F^{N q}(P)$ to $Q$ 
(the contracting direction at $Q$). So $\Gamma$ contains a straight 
line segment $S$ which expands under $F^{-q}$.
This expansion cannot continue indefinitely, so $F^{-mq}(I)$ 
must meet some $U_i$ (or $V_i$) for some $m$. 
But then this $U_i$ (or $V_i$) cannot be periodic, which contradicts the assumption. 
\end{proof}

\section{Special parameter values}\label{sec:spec_vals}

In this section, we will show that for a certain countable 
subset of parameter values $\theta \in (0,\frac{\pi}{4})$, 
the map $F = F_\theta$ has invariant circles of the form described in the previous section. 

\begin{theorem}\label{thm:spec_params}
 There exists a sequence of parameter angles $\theta_3, \theta_4, \ldots \in (0,\frac{\pi}{4})$, 
$\theta_K \to \frac{\pi}{4}$ as $K \to \infty$, such 
that for each $K$, $F_{\theta_K}$ has a countable collection 
of invariant circles $\{\Gamma_K^N : N \geq 0 \}$, each of rational 
rotation number $\rho(\Gamma_K^N) = 1/(4(K+N))$.
The curves $\Gamma_K^N$ consist of straight line segments, 
are rotationally symmetric and converge to the boundary
$\partial \rect$ as $N \to \infty$. 
\end{theorem}

\begin{proof}
We will prove the result by explicitly finding periodic orbits 
in $\bigcup_i (\mca_i \cup \mcb_i)$, which
hit the break lines whenever passing from one of the pieces to another.
More precisely, we will show that for $K \geq 3$ and $N \geq 0$, 
there is a parameter value $\theta = \theta_K$ and a point $X_N \in \overline{E_1 P_4}$, 
such that $F^n(X_N) \in \mcb_1$ for $0 \leq n < K$, $F^K(X_N) \in \overline{E_1 P_1}$, 
$F^n(X_N) \in \mca_1$ for $K \leq n < K+N$, and $F^{N+K}(X_N) = R(X_N) \in \overline{E_2 P_1}$, 
where $R$ is the counterclockwise rotation 
by $\frac{\pi}{2}$ about the centre of the square, see Fig.\ref{fig:per_orbit}.
By symmetry this clearly gives a periodic cancellation orbit, 
and an invariant circle is then given by 
\[\Gamma_K^N = \bigcup_{i=0}^{4(K+N)} \overline{F^i(X_N) F^{i+1}(X_N)}.\]

\begin{figure}
\centering
\includegraphics[width=\textwidth]{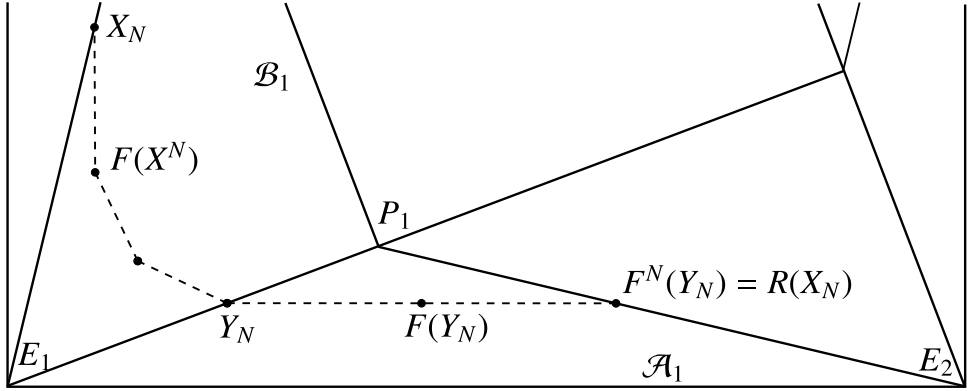}
\caption{Construction of (a part of) a periodic 
cancellation orbit in Theorem \ref{thm:spec_params} for $K = 3$, $N = 2$.
The points $X_N \in \overline{E_1 P_4}$ and $Y_N \in \overline{E_1 P_1}$ 
are chosen such that $F^K(X_N) = Y_N$ (Lemma \ref{lem:B_int_it})
and $F^N(Y_N) = R(X_N) \in \overline{E_2 P_1}$ (Lemma \ref{lem:A_int_it}).
The dots are the $F$-iterates of $X_N$, the dashed lines indicate 
the line segments making up (part of) the invariant circle $\Gamma_K^N$.}
\label{fig:per_orbit}
\end{figure}

We use the following two lemmas, whose proofs we leave to the appendix.

\begin{lemma}\label{lem:B_int_it}
 For every $K \in \N$, $K \geq 3$, there exists a parameter 
$\theta_K \in (0,\frac{\pi}{4})$, such that 
for $F = F_{\theta_K}$, $F^k(\overline{E_1 P_4}) \subset \mcb_1$ for $0 \leq k < K$, 
and $F^K(\overline{E_1 P_4}) = \overline{E_1 P_1} = \mcb_1 \cap \mca_1$. 
For $K \to \infty$, the angle $\theta_K$ tends to $\frac{\pi}{4}$.
\end{lemma}

\begin{lemma}\label{lem:A_int_it}
 For every $\theta \in (0,\frac{\pi}{4})$ and $N \geq 0$ there exists a point 
$Y_N \in \overline{E_1 P_1} = \mcb_1 \cap \mca_1$ such that
$F^n(Y_N) \in \mca_1$ for $0 \leq n < N$ and 
$F^N(Y_N) \in \overline{E_2 P_1} = \mca_1 \cap \mcb_2$.
For $N \to \infty$, the points $Y_N$ converge to $E_1$.
\end{lemma}

By Lemma \ref{lem:B_int_it}, for every $K \geq 3$ we can find 
$\theta_K$, such that $F^K$ maps $\overline{E_1 P_4}$ 
to $\overline{E_1 P_1}$ (in $\mcb_1$). 
Then by Lemma \ref{lem:A_int_it}, for any $N \geq 0$ there exists
$X_N \in \overline{E_1 P_4}$ such that 
$F^K(X_N) = Y_N \in \overline{E_1 P_1}$ and $F^{K+N}(X_N) \in \overline{E_2 P_1}$,
and it can be seen that each open line segment 
$(F^k(X_N), F^{k+1}(X_N))$, $k = 0,\ldots,K+N-1$, 
lies in the interior of either $\mcb_1$ ($0 \leq k < K$) 
or $\mca_1$ ($K \leq k < K+N$), see Fig.\ref{fig:per_orbit}.

Further, $B$ preserves the quadratic form $Q_B(x,y) = t (x^2 + y^2) - x y$.
With $X_N = (x_1,y_1) \in \overline{E_1 P_4}$ 
and $F^K(X_N) = (x_2, y_2) \in \overline{E_1 P_1}$ as above, 
a simple calculation shows that $Q_B(X_N) = Q_B(F^K(X_N))$ implies that $x_1 = y_2$. 
Then, with $F^{K+N}(X_N) = (x_3,y_3) \in \overline{E_2 P_1}$, one has $y_3 = y_2$, as
$A$ preserves the y-coordinate. 
We have that $R(\overline{E_1 P_4}) = \overline{E_2 P_1}$, and since $x_1 = y_3$, 
it follows that $F^{K+N}(X_N) = R(X_N)$. 
Rotational symmetry (Lemma \ref{lem:rot_sym}) then implies
that $X_N$ is a periodic point of period $4(K+N)$ for $F$, 
and the line segments connecting the successive $F$-iterates of $X_N$
form a rotationally symmetric invariant circle for $F$. 

Hence we obtain, for each $K \geq 3$, a parameter 
$\theta_K$ such that $F_{\theta_K}$ has a sequence
of rotationally invariant circles $\Gamma_K^N$, 
$N \geq 0$, each consisting of straight line segments.
Moreover, Lemma \ref{lem:A_int_it} implies that
$\Gamma_K^N \to \partial \rect$ (in the Hausdorff metric)
and $\rho_{\Gamma_K^N} = 1/(4(K+N)) \to 0$ as $N \to \infty$.
\end{proof}

In the proof of the theorem, for a sequence 
of special parameter values $\theta_K$
we constructed periodic orbits hitting the 
break lines and invariant circles made up 
of line segments connecting the points of the periodic orbits.
A closer look at the behaviour of $F$ 
\emph{between} any two consecutive invariant 
circles in this construction reveals that 
in fact the dynamics of $F$ in these regions 
is very simple for these parameter values. 

To see this, let $\theta = \theta_K$, $K \geq 3$, $N \geq 0$,  
and take $X \in \Gamma_K^N \cap \overline{E_1 P_1}$, 
$Y \in \Gamma_K^{N+1} \cap \overline{E_1 P_1}$. 
Then $X$ and $Y$ are periodic cancellation orbits of 
periods $4(K+N)$ and $4(K+N+1)$ on the respective invariant circles. 
Let $\mathcal{R}$ be the quadrangle with vertices $X, Y, F(Y), F(X)$. 
Then $F^{K+N}(X) = R(X)$ and $F^{K+N+1}(Y) = R(Y)$ 
lie in $\overline{E_2 P_2} = R(\overline{E_2 P_2})$, 
and it is easy to see that $F^{K+N}$ maps the triangle 
$D_1 = \Delta(X,F(Y),F(X))$ affinely to the triangle
$R(\Delta(X,Y,F(X)))$ and $F^{K+N+1}$ maps $D_2 = \Delta(X,Y,F(Y))$ 
affinely to $R(\Delta(Y,F(Y),F(X)))$.
This gives a piecewise affine map 
$\Phi \colon D_1 \cup D_2 = \mathcal{R} \to R(\mathcal{R})$ 
of the simple form shown in Fig.\ref{fig:returnmap}. 
By symmetry, the first return map of $F$ to the quadrangle $\mathcal{R}$ is then the fourth 
iterate of such map, and is easily seen to preserve the y-coordinate.

\begin{figure}
\centering
\includegraphics[width=0.6\textwidth]{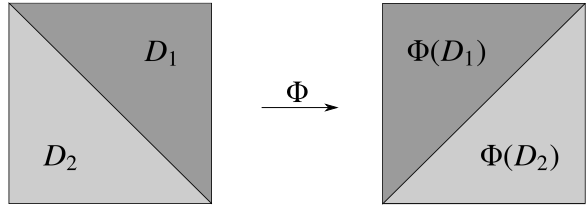}
\caption{The form of the map $\Phi \colon \mathcal{R} \to R(\mathcal{R})$, where $\mathcal{R}$
is the quadrangle formed by two consecutive cancellation orbit points on 
each of two adjacent invariant circles $\Gamma_K^N$, $\Gamma_K^{N+1}$. The first return map 
of $F$ to $\mathcal{R}$ is (by symmetry) the fourth iterate of $\Phi$.}
\label{fig:returnmap}
\end{figure}

It follows immediately that in fact all points 
in the annulus between $\Gamma_K^N$ and $\Gamma_K^{N+1}$
lie on invariant circles. These invariant circles are of the same form as the $\Gamma_K^N$, 
that is, they consist of line segments parallel to those explicitly 
constructed in the proof of Theorem \ref{thm:spec_params}. 
The rotation number of the invariant circle
through the point $W \in \overline{E_1 P_1}$ changes continuously (in fact, linearly) 
from $1/(4(K+N))$ to $1 / (4(K+N+1))$, as $W$ goes from $X$ to $Y$.

These invariant circles take on both rational and irrational rotation numbers, 
and their intersections with the break lines 
are not necessarily periodic as those of the $\Gamma_K^N$,
but their orbits still form cancellation orbits, 
since $F^K(\overline{E_1 P_4}) = \overline{E_1 P_1}$ (see Remark \ref{rem:canc_orbits}).
We get the following corollary.

\begin{corollary}\label{cor:inv_circ_foliation}
 For $\theta = \theta_K$, $K \geq 3$, as in 
Theorem \ref{thm:spec_params}, the annulus between $\partial \rect$
and $\Gamma_K^0$ (the invariant circle containing $P_i$ and $P_i'$, $i \in \irange$) is completely foliated 
by rotationally symmetric invariant circles with rotation numbers 
continuously and monotonically varying from $0$ on $\partial \rect$ to  $1/(4K)$ on $\Gamma_K^0$.
\end{corollary}

\begin{remark}\label{rem:firstparam}
 One can check that $K = 3$ in the proof of Theorem \ref{thm:spec_params} is obtained by 
setting $\theta = \frac{\pi}{8}$, corresponding to $t = (2-\sqrt 2) /4$ and $s = \sqrt 2 /4$.
This is the case when $C$ is the rotation by $\frac{\pi}{4}$ on $\mcc$.
For $K \geq 4$, exact values for $\theta$ are less easy to determine explicitly. 

In this special case $\theta = \frac{\pi}{8}$, the map $F$ turns out to be of a 
very simple form, allowing
a complete description of the dynamics on all of $\rect$. 
By a similar argument applied to the region inside 
$\Gamma_3^0$ (containing the rotational part $\mcc$), the statement of 
Corollary \ref{cor:inv_circ_foliation} can then be strengthened, 
stating that in this case
the whole space $\rect$ is foliated by invariant circles, 
with rotation numbers varying continuously 
and monotonically from $0$ on $\partial \rect$ to 
$\frac{1}{8}$ on the invariant octagon $\mathcal{O}$ inscribed 
in $\mcc$. The invariant circles between $\Gamma_3^0$ and 
$\mathcal{O}$ each consist of twelve straight line segments, 
parallel to the twelve segments of $\Gamma_3^0$, 
%of which four decrease in length to zero, as the circles 
%approach $\partial \mathcal{O}$, 
see Fig.\ref{fig:firstparam}.  
\end{remark}

\begin{figure}
\centering
\includegraphics[width=0.5\textwidth]{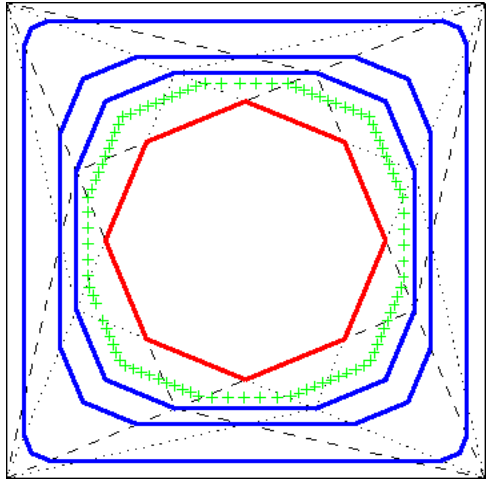}%,trim = 0 200 0 200
\caption{The case $\theta = \theta_3 = \frac{\pi}{8}$ from Remark \ref{rem:firstparam}. 
The solid piecewise straight line circles are 
the invariant octagon $\mathcal{O}$ inscribed in $C$, as well as the 
invariant circles $\Gamma_3^0$, $\Gamma_3^1$ and $\Gamma_3^{10}$. In the annulus 
between $\mathcal{O}$ and $\Gamma_3^0$, the first 100 iterates of an orbit are depicted, 
all lying on an invariant circle which consists of twelve straight line segments.}
\label{fig:firstparam}
\end{figure}

\section{General parameter values and discussion}\label{sec:gen}

%To end this section, let us briefly mention more general \q{higher order} cancellation orbits and 
%invariant circles.
For other parameter values than the ones considered in the previous section, 
we generally cannot prove the existence of any invariant circles. 
Let us briefly mention more general \q{higher order} cancellation orbits and 
invariant circles. 

Recall that in Theorem \ref{thm:spec_params}, we constructed a family of invariant circles 
consisting of line segments connecting successive orbit points of periodic cancellation orbits.
The chosen cancellation orbits were of the simplest possible kind, 
where a point on any break line is mapped by a certain number of iterations to the next possible break line.

%It is conceptionally not much more difficult but certainly more tedious
It is possible to construct invariant circles from one or several
more complicated periodic cancellation orbits\footnote{Due to the rotational symmetry of the system,
a periodic cancellation orbit could contain one, two, or four pairs of break points. 
In the first two cases, the union of all rotated 
copies of the cancellation orbit would need to be considered, 
to form an invariant circle by adding straight 
line segments between successive points in this union.}
(for other values of $\theta$ than the ones in Theorem \ref{thm:spec_params}). 
On such periodic cancellation orbit, 
the iterates of a point on a break line would cross several break lines before 
hitting one. 
The resulting invariant circle would still consist of straight line segments, but 
in such a way that a given segment and its image under $F$ are not adjacent on the circle.

Doing this \q{higher order} construction is conceptionally not much more difficult, 
but certainly more tedious than the \q{first order} construction of Theorem \ref{thm:spec_params}. 
The effort to construct even just a single higher order periodic cancellation orbit 
seems futile, unless a more general scheme to construct all or many of them at once can be found. 

It is unclear whether invariant circles of this type exist for all $\theta$, and whether 
for typical $\theta$ there exist piecewise line segment invariant circles 
(of rational or irrational rotation number) containing non-periodic cancellation orbits, 
as the ones seen in Corollary \ref{cor:inv_circ_foliation}. 

\begin{question}
For which parameter values $\theta \in (0,\frac{\pi}{4})$ does $F_\theta$ have invariant 
circles with periodic cancellation orbits (and, therefore, rational rotation number)?
For which $\theta$ are there piecewise line segment 
invariant circles with non-periodic cancellation orbits?
\end{question}

\begin{figure}
\centering
\includegraphics[width=\textwidth]{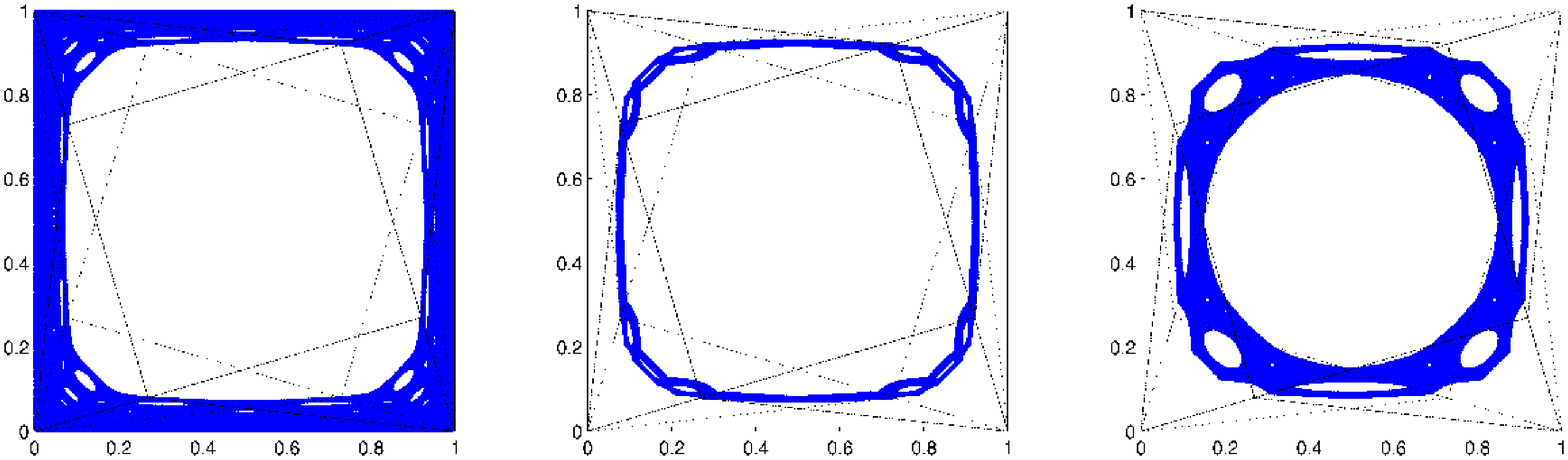}%,trim = 120 150 100 120,clip=true
\caption{The first $5\cdot10^5$ iterates of three initial points for $\theta = \frac{\pi}{11}$.
Each of the orbits seems to be confined to an invariant annulus and fill this annulus densely, except for a 
number of elliptic islands.}
\label{fig:erg_reg}
%\label{fig:erg_reg}
%\end{figure}
%\begin{figure}[h]
%\centering
\includegraphics[width=\textwidth,trim = 0 0 0 -30]{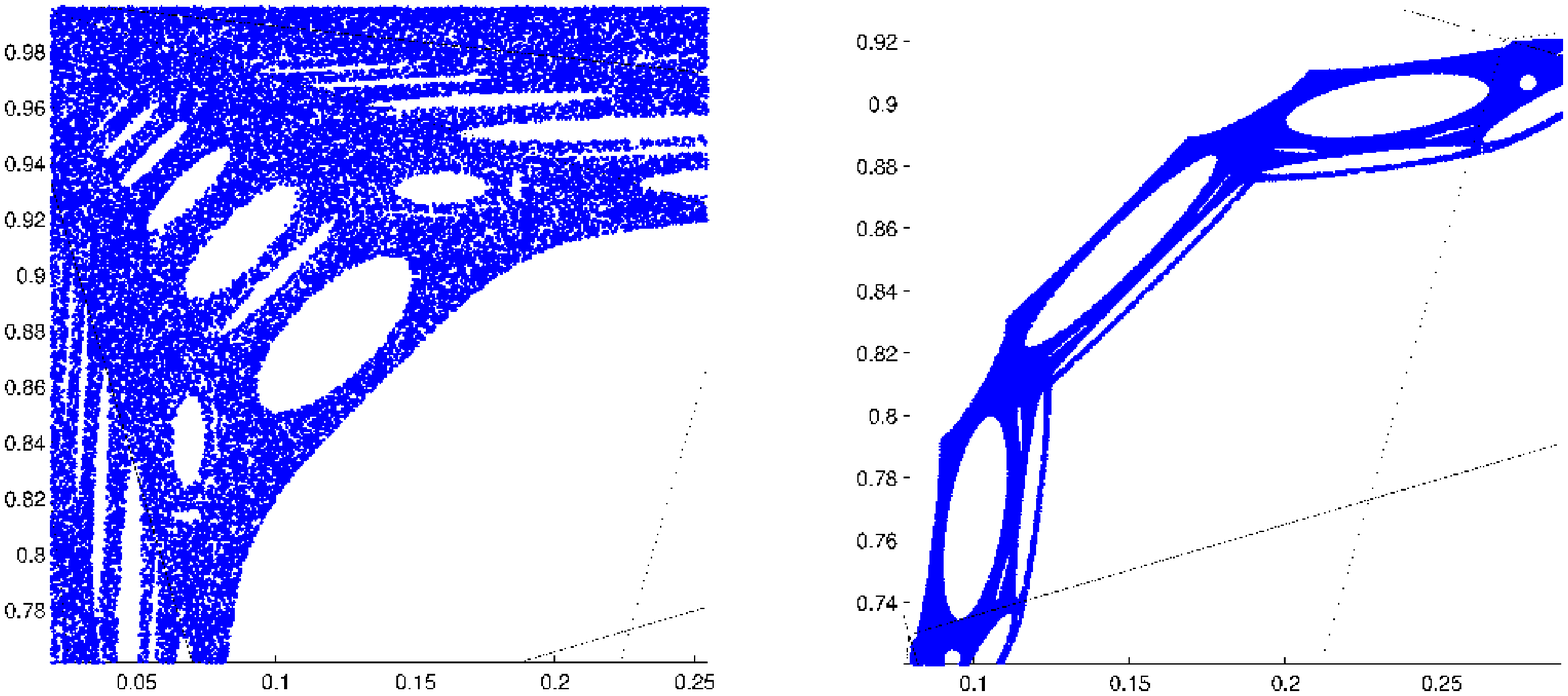}%,trim = 150 100 170 100,clip=true
\caption{Zoom-in of the first two orbits from Fig.\ref{fig:erg_reg}, 
showing that the invariant annuli contain
large numbers of smaller and smaller elliptic island chains. 
Further zoom-in (and longer orbits) reveal increasingly 
intricate patterns of such quasi-periodic elliptic regions.}
\label{fig:erg_reg_zoom}
\end{figure}

Moreover, we can at this point not rule out the existence of 
invariant circles (outside of $\mcc$) of an entirely different kind, 
not consisting of straight line segments. Indeed, some numerical experiments seem to indicate 
the occurence of invariant regions with smooth boundaries, but it is unclear whether this is 
due to the limited resolution (see, for example, left picture in Fig.\ref{fig:erg_reg} and Fig.\ref{fig:erg_reg_zoom}).

By Proposition \ref{prop:rot_num}, the intersections of any invariant circle of rational rotation number
with the break lines form cancellation orbits. For irrational rotation number, 
this need not be the case. Note, however, that under irrational rotation, 
the \q{kink} introduced to an invariant circle at 
its intersection with a break line propagates densely 
to the entire circle, unless it is cancelled by eventually being mapped 
to another break line intersection. 
Hence, an invariant circle  of irrational rotation number would either 
contain a cancellation orbit for each pair of break lines, 
or otherwise would be geometrically
complicated, namely nowhere differentiable.

\begin{question}
 Are there invariant circles for $F$ (outside of $\mcc$) which are not comprised
of a finite number of straight line segments? 
Are there invariant circles whose intersections
with the break lines do not form cancellation orbits?
\end{question}

%\medskip

As for the dynamics of $F$ between invariant circles, 
we can only point to numerical 
evidence that annuli between consecutive invariant circles 
form ergodic components interspersed with 
\q{elliptic islands}. An elliptic island consists of a periodic point of, 
say, period $p$, surrounded by a family 
of ellipses which are invariant under $F^p$, such that $F^p$ 
acts as an irrational rotation on each of these ellipses; this is referred to 
as \q{quasi-periodic} behaviour.
In the case when $F^p$ is a rational rotation, these quasi-periodic 
invariant circles would take on the shape of polygons, consisting entirely of $p$-periodic points.
The rest of the annuli seems to be filled with what is often 
referred to as \q{stochastic sea}, that is, the dynamics seems to be ergodic 
and typical orbits seem to fill these regions densely.

\begin{figure}
\centering
\includegraphics[width=\textwidth]{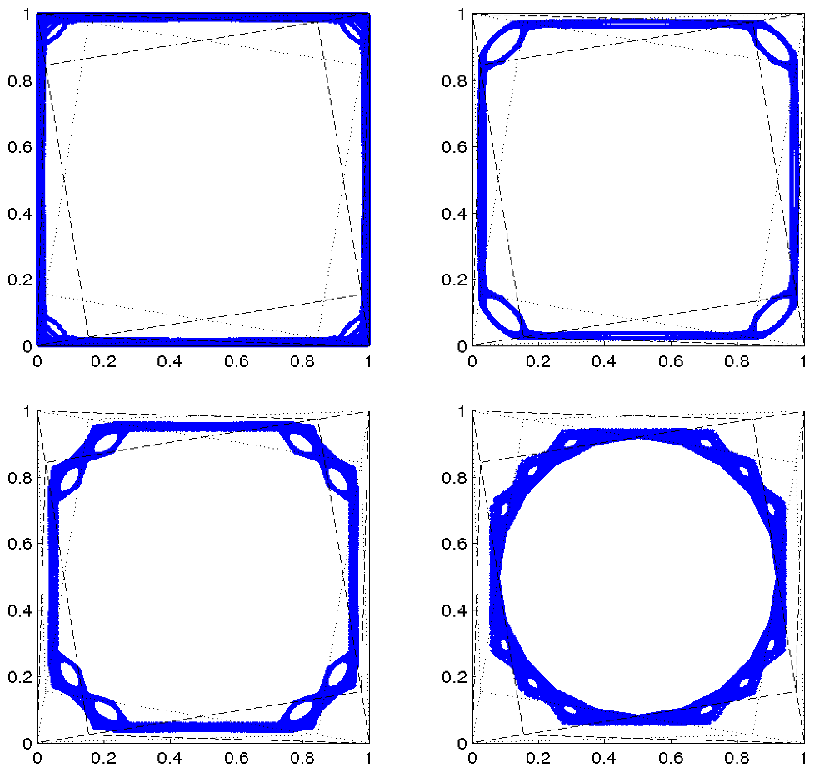}
\caption{The first $10^5$ iterates of four initial points for $\theta = \frac{\pi}{20}$.
Each seems to be a dense orbit in an invariant annulus. The four annuli (without a number of elliptic islands and 
the invariant periodic regular 20-gon inscribed in $\mcc$) seem to be the ergodic components of $F$.}
\label{fig:45pi}
\end{figure}

\begin{figure}
\centering
\includegraphics[width=\textwidth]{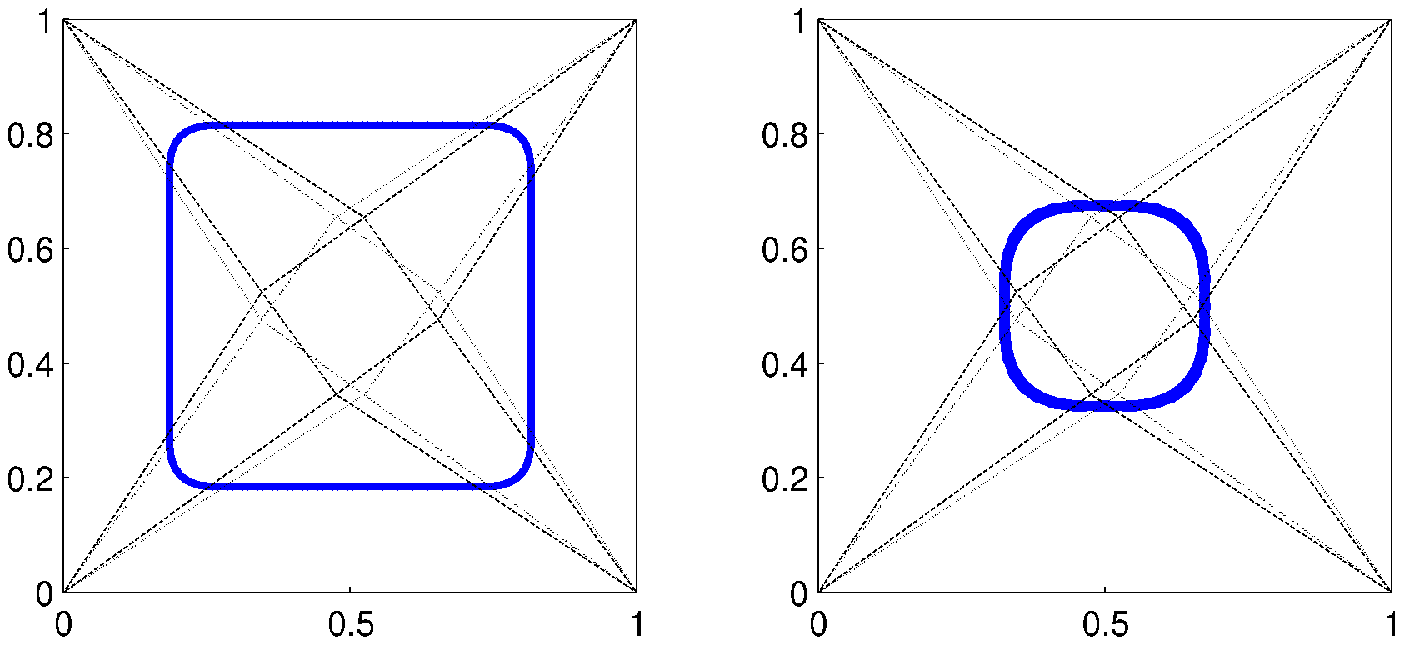}
\caption{The first $10^5$ iterates of two initial points for $\theta = \frac{\pi}{5}$.
Each of the orbits seems to densely fill a thin invariant annulus. The rectangle seems 
to be partitioned into finitely many such invariant annuli 
(which get thinner and more numerous as $\theta \to 0$).}
\label{fig:pi_5}
\includegraphics[width=\textwidth,trim = 0 0 0 -30]{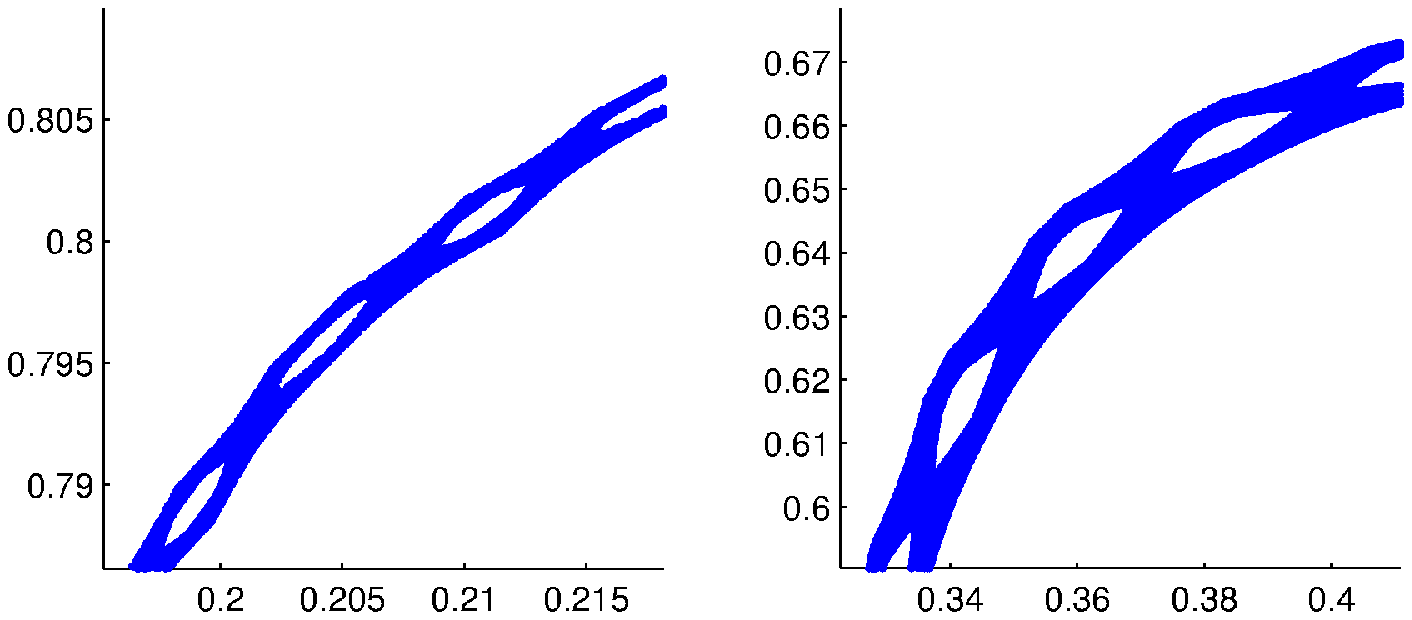}
\caption{Zoom-in of the orbits from Fig.\ref{fig:pi_5}. The thin invariant annuli 
contain periodic island chains, which under even stronger magnification 
could be seen to be surrounded by further, finer, islands of quasi-periodic motion.}
\label{fig:pi_5zoom}
\end{figure}

As in many similar systems (e.g., perturbations of the standard map), 
numerical observations seem to indicate that these \q{ergodic regions}
have positive Lebesgue measure (see Figures \ref{fig:erg_reg}-\ref{fig:pi_5zoom}).
This is related to questions surrounding the famous 
\q{quasi-ergodic hypothesis}, going back to Ehrenfest \cite{Ehrenfest} and Birkhoff \cite{Birkhoff1932}, 
conjecturing that typical Hamiltonian dynamical systems have dense orbits 
on typical energy surfaces (see also \cite{Herman1998}).

For a piecewise linear version of the standard map, Bullett \cite{Bullett1986} established 
a number of results on cancellation orbits and invariant circles of both rational 
and irrational rotation numbers. Wojtkowski \cite{Wojtkowski1981,Wojtkowski2008} 
showed that the map is almost hyperbolic and deduced that it has ergodic components of positive
Lebesgue measure. Almost hyperbolicity here means the almost everywhere 
existence of invariant foliations of the space (or an invariant subset) 
by transversal local contracting and expanding fibres. Equivalently, 
this can be expressed through the existence of invariant contracting and expanding cone fields. 
By classical theory (see, for example, \cite{Sinai}) almost hyperbolicity
implies certain mixing properties of the map, and,
in particular, the existence of an at most countable collection of ergodic components of positive 
Lebesgue measure. 

A similar kind of cone construction as in \cite{Wojtkowski1981,Wojtkowski2008}
seems to be more difficult for the map studied in this paper. 
One important difference is that the piecewise linear standard map in these 
papers is a twist map, which is not strictly the case for the map $F$ studied here (see \cite{LeCalvez2000}
and references therein for an overview over the numerous classical results for twist maps, 
mostly based on Birkhoff and Aubry-Mather theory).
The additional property that $F$ equals the identity on the boundary of the square also sets 
it apart. In particular, the motion of points under $F$ close to the boundary 
can be arbitrarily slow, that is, take arbitrarily many iterations to pass through the piece 
$\mca_i$ (while the number of iterations for a passage through $\mcb_i$ remains bounded). 
This seems to make it more difficult to explicitly construct an invariant contracting 
or expanding cone field, as was done for the piecewise linear standard map. 
Moreover, such an invariant cone field construction can not be carried out 
uniformly for all $\theta$. 
In fact, as can be seen from Corollary \ref{cor:inv_circ_foliation}, there are parameter values 
$\theta_K$, $K=3,4,\ldots$,
for which almost hyperbolicity cannot hold on large parts of $\rect$, 
as the dynamics is completely integrable on the annulus between the invariant
circle $\Gamma_K^0$ and $\partial \rect$ (and even on all of $\rect$ for 
$\theta = \theta_3 = \frac{\pi}{8}$, see Remark \ref{rem:firstparam}).

We are led to leave the following as a question.

\begin{question}
 Are there parameter values $\theta$ for which the map $F_\theta$ is almost hyperbolic
on some invariant subset of $\rect$? How large is the set of parameters $\theta$ for which this is the case?
\end{question}

While it does not seem likely that almost hyperbolicity can be shown for almost all $\theta$, 
numerical evidence suggests that for typical $\theta$, the map $F_\theta$ has a finite number
of ergodic components of positive Lebesgue measure, in which typical points have dense orbits.

\begin{conjecture}
 For Lebesgue almost all $\theta \in (0,\frac{\pi}{4})$, there is a finite number of $F$-invariant sets 
 $A_1,\ldots,A_m$, each of positive Lebesgue measure, such that $F\vert_{A_i} \colon A_i \to A_i$ is 
ergodic for every $i=1,\ldots,m$. Each $A_i$ is a topological annulus
with a certain number of elliptic islands removed from it, 
and, together with the elliptic islands
and the invariant disk inscribed in $\mcc$, the $A_i$ form a partition of $\rect$.
\end{conjecture}

Numerical experiments also seem to indicate that for many parameter values,
the way in which chaotic and (quasi-)periodic behaviour coexist,
that is, the structure of invariant annuli 
containing families of quasi-periodic elliptic islands, can be quite rich, 
see Fig.\ref{fig:erg_reg} and Fig.\ref{fig:erg_reg_zoom}.
Besides the total measure of such quasi-periodic elliptic islands, 
it would be also interesting to know whether a general scheme 
for their itineraries, periods and rotation numbers can be found.

\section*{Acknowledgements}
I would like to thank my PhD supervisor, Sebastian van Strien, for the numerous helpful discussions
while I was doing the work presented in this paper. I am also grateful to Shaun Bullett for
discussing with me his early work on the piecewise linear standard map and 
making me aware of some of the references.

\section*{Appendix. Proofs of Lemmas \ref{lem:B_int_it} and \ref{lem:A_int_it}}

\begin{proof}[Proof of Lemma \ref{lem:B_int_it}]
First, recall that the map $B\colon \mcb_1 \to \mcb_1 \cup \mca_1 \cup \mcc$ leaves invariant 
the quadratic form (\ref{eq:quad_form_B}). Then, using $P_4 = (1-s, t)$ and $P_1 = (s, t)$, we 
calculate
\[
Q_B(P_4) - Q_B(P_1) = 
\left((1-s)^2 + t^2 - \frac{(1-s)t}{t}\right) - \left(s^2 + t^2 - \frac{st}{t}\right) = 0,
\]
and hence $Q_B(P_4) = Q_B(P_1) =: c$. 
So the line $\{x \in \mcb_1 : Q_B(x) = c \}$ is a segment of a hyperbola connecting $P_4$ and $P_1$. 
Therefore, with $V = \{ x \in \mcb_1 : Q_B(x) \leq c \}$, we get $F(V) = B(V) \subset (V \cup \mca_1)$. 

Further, note that $B$ maps rays through $E_1$ to other rays through $E_1$. 
This implies that all points on the straight line segment $\overline{E_1 P_4} = \mca_4 \cap \mcb_1$ 
remain in the piece $\mcb_1$ for equally many 
iterations of the map $F$, before being mapped into $\mca_1$.
In particular, if for $X \in \overline{E_1 P_4}$ we have that 
\begin{equation}\label{eq:B_int_it}
 F^k (X) \in 
\begin{cases} \mcb_1 & \text{if } 0 \leq k < K,\\
\overline{E_1 P_1} & \text{if } k = K ,
\end{cases}
\end{equation}
then the same holds for every other point $X' \in \overline{E_1 P_4}$, and $F^K(\overline{E_1 P_4}) = \overline{E_1 P_1}$.

We will now show that there exists a sequence of parameter values $\theta_K$, $K \geq 3$,
such that for $F= F_{\theta_K}$, $F^k(P_4) = B^k(P_4) \in \mcb_1$ for $0 \leq k \leq K$
and $F^K(P_4) = P_1$.
For this, we need a few elementary facts about the map $B$, 
which follow from straightforward (but rather tedious) calculations:
\begin{itemize}
 \item Let $f(t) = \sqrt{1-4t^2}$. Then the hyperbolic map $B$ has two eigendirections 
\[
 v_1 = \begin{pmatrix} 1 + f(t) \\ 2t \end{pmatrix},\quad v_2 = \begin{pmatrix} 2t \\ 1 + f(t) \end{pmatrix}
\]
with corresponding eigenvalues
\begin{equation}\label{eq:eig_vals}
 \lambda_1 = \lambda = \frac{4t^2 - 2t + (2s - 1) (1+f(t))}{2(t-s)} > 1, \quad 
 \lambda_2 = \lambda^{-1} < 1.
\end{equation}
\item By a linear change of coordinates $\Phi \colon \R^2 \to \R^2$ 
mapping $v_1$ and $v_2$ to $(1, 0)$ and $(0,1)$, respectively, one gets a conjugate linear map
\[
 \tilde B = \Phi \circ B \circ \Phi^{-1} 
= \begin{pmatrix} \lambda & 0 \\ 0 & \lambda^{-1} \end{pmatrix}.
\]
Setting $\Phi ( P_4 ) = \Phi (t, 1-s) =: (x_1, y_1)$ and $\Phi ( P_1 ) = \Phi (s, t) =: (x_2, y_2)$,
a somewhat tedious calculation gives 
\begin{equation}\label{eq:frac_pts}
 \frac{x_2}{x_1} = \frac{2ts + t(f(t)-1)}{2t^2 + (1-s)(f(t)-1)}.
\end{equation}
\end{itemize}

Now, since $\tilde B^K = \Phi \circ B^K \circ \Phi^{-1}$, 
it follows that $B^K(P_4) = P_1$ is equivalent to $\tilde B^K \Phi ( P_4 ) = \Phi (P_1)$. 
By the simple form of $\tilde B^K$, this is equivalent to $x_2 = \lambda^K x_1$, that is,
\begin{equation}\label{eq:K}
 K = \frac{\log(x_2 / x_1)}{\log(\lambda)}.
\end{equation}

Substituting (\ref{eq:eig_vals}) and (\ref{eq:frac_pts}) into (\ref{eq:K}), and 
using $s = \sqrt{t - t^2}$ from (\ref{eq:relations}), we get an expression $K = K(t)$ for $0 < t < \frac{1}{2}$. 
Then $K$ is differentiable and strictly monotonically decreasing as a function of $t$, 
and (by application of L'H\^{o}pital's rule)
\[
 K \to \begin{cases} 2 & \text{ as } t \to 0, \\ 
\infty & \text{ as } t \to \frac{1}{2}.\end{cases}
\]
Since $\sin^2 \theta = t$ with $t \in (0, \frac{1}{2})$, $\theta \in (0,\frac{\pi}{4})$, we get that
for each $K \geq 3$ there exists $\theta_K \in (0,\frac{\pi}{4})$ such that $B^K(P_4) = P_1$, 
hence $B^K(\overline{E_1 P_4}) = \overline{E_1 P_1}$, as claimed. 
\end{proof}

\begin{proof}[Proof of Lemma \ref{lem:A_int_it}]
First, the shear map $A \colon \mca_1 \to \mca_1 \cup \mcb_2$ is such that any point
$(x,y) \in \mca_1$ is mapped to $F(x,y) = A(x,y) = (x + \tilde{c} y, y)$, $\tilde c > 0$.
By continuity, it follows that for any $N \geq 0$ there exists a point 
$Y_N \in \overline{E_1 P_1} = \mcb_1 \cap \mca_1$ such that
$F^n(Y_N) \in \mca_1$ for $0 \leq n < N$ and $F^N(Y_N) \in \overline{E_2 P_1} = \mca_1 \cap \mcb_2$. 
Clearly, $Y_0 = P_1 = (s, t)$, and one can calculate that
\[
Y_N = \frac{1}{1 + N(1-2s)} (s,t), \quad N \geq 0,
\]
and $Y_N \to (0,0) = E_1$ as $N \to \infty$.
\end{proof}

\bibliographystyle{abbrv}
\bibliography{../BibTex/library}

\end{document}